\documentclass[12pt]{amsart}

\usepackage{amsmath}
\usepackage{amssymb}
\usepackage{amsfonts}
\usepackage{amsthm}
\usepackage{enumerate}
\usepackage{hyperref}
\usepackage{color}
\usepackage[all]{xy}
\usepackage{tikz}
\usepackage{comment}

\theoremstyle{plain}
\newtheorem{thm}{Theorem}[section]

\newtheorem{prop}[thm]{Proposition}
\newtheorem{lem}[thm]{Lemma}
\newtheorem{cor}[thm]{Corollary}

\theoremstyle{remark}

\theoremstyle{definition}
\newtheorem{cons}[thm]{Construction}
\newtheorem{dfn}[thm]{Definition}
\newtheorem{rem}[thm]{Remark}
\begin{document}
\title{Cellular structure for the Herzog--Takayama Resolution}

\author{Afshin Goodarzi}
\thanks{}

\address{Royal Institute of Technology, Department of Mathematics, S-100 44, Stockholm, Sweden}
\email{afshingo@kth.se}
\maketitle

\begin{abstract}
 Herzog and Takayama constructed explicit resolutions for the class of so called ideals with a regular linear 
 quotient. This class contains all matroidal and stable ideals. The resolutions of matroidal and stable ideals 
 are known to be cellular. In this note we show that the Herzog--Takayama resolution is also cellular.
\end{abstract}
\section{Introduction}
\noindent In this section we describe our problem. The exact definitions will appear in the next sections.
Also, we refer to the books by Hatcher~\cite{Ha}, Peeva~\cite{P} and Milller and Sturmfels~\cite{MS}, for 
undefined terminology.\\ 
The idea of determining free resolutions of monomial ideals in a polynomial ring using labelled chain 
complexes of topological objects was introduced by Bayer, Peeva and Sturmfels \cite{BPS} by considering the 
labelled chain complex of simplicial complexes. Every monomial ideal admits such a resolution, namely the Taylor 
resolution, although it is most of the time far from being minimal.  This method was generalized in \cite{BS}, 
where the authors constructed a regular cell complex (the hull complex) for any monomial ideal. The resolution
obtained from the hull complex is much shorter than the Taylor resolution in 
general. Batzies and Welker in \cite{BW} developed an algebraic analogue of Forman's discrete Morse theory to 
minimize the resolutions.\\
Unfortunately, even if we consider the more general case of CW complexes, as the authors did  in \cite{BW}, 
it is not the case that one can describe the minimal free resolution of a monomial ideal using these 
methods. A family of examples showing this is given in \cite{V}. However, for a monomial ideal with a linear quotient 
it is shown in \cite[Proposition 4.3.]{BW} that there exists a CW-complex which supports the minimal free 
resolution of the ideal. The minimal free resolution of a class of ideals with a linear quotient, the so 
called ideals with a \emph{regular linear quotient} (see Definition~\ref{de}), is determined by
Herzog and Takayama in \cite{HT}. This class contains all matroidal ideals, stable ideals and square-free 
stable ideals. In \cite{NPS} the authors have given a polytopal complex that supports the minimal free 
resolution of matroidal ideals. In \cite{M}, Mermin constructed a regular cell complex that supports the 
Eliahou-Kervaire resolution for stable ideals. On the other hand, since the square free part of the 
Eliahou-Kervaire resolution resolves square-free stable ideals minimally, such a construction for square-free stable 
ideals exists. So it is natural to ask whether the Herzog-Takayama resolution is cellular or not? The aim of 
this paper is to give a positive answer to the question above by explicitly constructing a regular cell 
complex that supports the Herzog-Takayama resolution for ideals with a regular linear quotient. 

\section{Monomial Ideals with a Linear Quotient and Cellular Resolutions}\label{s2}
\noindent Let $K$ be a field and $S=K[x_1,\ldots, x_n]$ be the polynomial ring in $n$ variables over $K$. For a monomial 
ideal $I$ in $S$, we denote by $M(I)$ the set of all monomials in $I$. We also denote by $G(I)$ the unique 
minimal set of generators of $I$. We say that $I$ has a \emph{linear quotient}, if $G(I)$ admits an 
\emph{admissible order}, that is a linear ordering $u_1,\ldots,u_m$ of monomials in $G(I)$ such that the colon ideal 
$\langle u_1,\ldots,u_{j-1} \rangle:u_j$ is generated by a subset $\mathfrak{q}(u_j)$ of variables 
for all $2\leq j\leq m$. To any admissible order of $I$ one can associate a unique 
\emph{decomposition function}, that is, a function $\mathfrak{g}:M(I)\rightarrow G(I)$ that maps a 
monomial $v$ to $u_j$, if $j$ is the smallest 
index for which $v\in I_j$, where $I_j:=\langle u_1,\ldots,u_{j} \rangle$ (see \cite{HT} for more on
decomposition functions).\\
It is not difficult to check that an ordering $u_1,\ldots,u_m$ of $G(I)$ is an 
admissible order if and only if it satisfies the following \emph{shelling type condition}, considered by 
Batzies 
and Welker \cite[Page 157]{BW}

\begin{equation}\label{shelling}
\begin{split}
 & \text{For all $j$ and $i<j$ there exists $k<j$ such that } \\
& \mbox{lcm}(u_k,u_j)=x_t\cdot u_j \text{ for some $x_t$ and $\mbox{lcm}(u_k,u_j)$}\\
& \text{divides $\mbox{lcm}(u_i,u_j)$.}
\end{split}
\end{equation}
 It is known \cite[Lemma 2.1]{JZ} that for a monomial ideal with a linear quotient there always 
exists a degree increasing admissible order. So, throughout this note all admissible orders
are considered to be increasing.\\
The following result was proved in \cite[Proposition 4.3]{BW} by using algebraic discrete Morse theory. 
Here, we give an alternative proof based on the observation that the resolutions obtained by iterated mapping 
cones are supported by CW complexes that can be constructed by iterated (topological) mapping cones. This is 
indeed the idea behind our construction in the next section.
\begin{prop}\label{mapping}
For any ideal $I$ with a linear quotient, there exists a CW complex $X_I$ that supports the minimal 
free resolution of $S/I$. 
\end{prop}
\begin{proof}
Let $u_1, \ldots , u_m$ be a degree increasing admissible order of $G(I)$. Assume that we have inductively 
constructed a CW complex $X_{j-1}$ that supports the minimal free resolution of $S/I_{j-1}$. Also, let $\Delta$ be the simplex associated to the Koszul 
resolution of  $S/\langle \mathfrak{q}(u_j)\rangle$. Lifting the left non-zero map of the short exact 
sequence 
\begin{equation*}
0\longrightarrow S/\langle \mathfrak{q}(u_j)\rangle \longrightarrow S/I_{j-1} \longrightarrow S/I_{j} 
\longrightarrow 0,
\end{equation*}
to the minimal free resolutions of $S/\langle \mathfrak{q}(u_j)\rangle$ and $S/I_{j-1}$, then induces a 
cellular map of $\Delta$ into $X_{j-1}$. Therefore the resolution of $S/I_j$ obtained by the 
homological mapping cone has a cellular structure that is the topological mapping cone of the cellular map of 
$\Delta$ into $X_{j-1}$. The fact that this resolution is minimal follows from \cite[Lemma 1.5]{HT}
\end{proof}
\begin{rem}
 Reiner and Welker~\cite[Section 5]{RW} gave an example of an ideal with a linear quotient for which the 
 differential matrices in the minimal free resolution cannot be written using only $\pm 1$ coefficients. 
 This shows that the cell complex $X_I$ in Proposition~\ref{mapping} is not regular in general.
\end{rem}

 The difficulty of constructing the cell complex $X_I$ in practice is then how to define 
the cellular 
maps in each step, which is indeed as difficult as defining the comparison maps $S/\langle 
\mathfrak{q}(u_j)\rangle \rightarrow S/I_{j-1}$. This has already been observed by Herzog and Takayama in 
\cite{HT}, where the authors defined a subfamily of ideals with a linear quotient that have a decomposition 
function with a similar behaviour as the stable ideals'.\\
\begin{dfn}\label{de} A decomposition function 
$\mathfrak{g}$ is said to be \emph{regular}, if $\mathfrak{q}(\mathfrak{g}(yu_j))\subseteq \mathfrak{q}(u_j)$, 
for any $j$ and any $y\in\mathfrak{q}(u_j)$. We say that $I$ has a \emph{regular linear quotient}, if it admits 
an admissible order with a regular decomposition function.
\end{dfn}
  The following equation for the decomposition 
function of ideals having regular linear quotient is proved in \cite{HT}, we will frequently use it
\begin{equation}\label{commute}
 \mathfrak{g}(y\mathfrak{g}(zu))=\mathfrak{g}(z\mathfrak{g}(yu))\quad \forall u\in G(I) 
 \quad \forall z,y\in \mathfrak{q}(u).
\end{equation}
It can be easily checked that the set of ideals with a regular linear quotient contains all stable and 
square-free stable ideals. On the other hand it is known \cite[Theorem 1.10]{HT} that the reverse 
lexicographic order on minimal generators of a matroidal ideal is an admissible order with a regular 
decomposition function.

\begin{dfn}
Let $\sigma$ be a subset of $V=\{x_1,\ldots,x_n\}$ and assume the total ordering $x_1<\ldots<x_n$ on $V$. Then
for an element $y\in\sigma$ set $\alpha(\sigma;y)$ to be the number of variables $z$ in $\sigma$ such that $z<y$. 
\end{dfn}
\begin{cons}
Let $I$ be a monomial ideal with a regular linear quotient. The \emph{Herzog-Takayama} resolution $\mathbf{F}_I$ of $S/I$ has basis denoted

$$\mathcal{B}=\{1\}\cup\{f(\sigma; u)\mid u\in G(I)\text{ and }\sigma\subseteq \mathfrak{q}(u) \},$$
where the element $f(\sigma; u)$ has homological degree $\mid \sigma\mid +1$. If $\sigma$ is non-empty, then we define 

$$\mu(f(\sigma;u))=\sum_{y\in\sigma} (-1)^{\alpha(\sigma;y)}yf(\sigma-\{y\};u),$$
and
$$\delta(f(\sigma;u))=\sum_{y\in\sigma} (-1)^{\alpha(\sigma;y)}\frac{yu}{\mathfrak{g}(yu)}f(\sigma-\{y\};\mathfrak{g}(yu)),$$
by the convention that $f(\tau;v)=0$, if $\tau$ is not a subset of $\mathfrak{q}(v)$. The differential in $\mathbf{F}_I$ is given by $\partial=\delta-\mu$, in homological degrees $>1$ and otherwise $\partial(f(\emptyset,u))=u$. 
\end{cons}\qed

\begin{thm}{\cite[Theorem 1.12]{HT}}
Let $I$ be a monomial ideal with a regular linear quotient. Then the Herzog-Takayama resolution $\mathbf{F}_I$ is the minimal free resolution of $S/I$.
\end{thm}

\section{Monomial Ideals with a Regular Linear Quotient and Their Associated Regular Cell Complexes}\label{s3}
We now associate a regular cell complex to an ideal with a regular linear quotient as follows: 

 \begin{cons}\label{cons}
  Let $I$ be a monomial ideal with a regular linear quotient with respect 
 to the admissible order $u_1,\ldots,u_m$ of $G(I)$. Also let $\mathfrak{g}$ be its decomposition function. 
 We will construct a regular labelled cell complex $X_I$ inductively, as follows:
 \begin{enumerate}
  \item[(i)] Let $X_1$ be the $0$-simplex with the labelled vertex $\{u_1\}$,
  \item[(ii)] assume that the regular labelled cell complex $X_{j-1}$ with vertices $u_1,\ldots,u_{j-1}$ 
  is constructed,
  \end{enumerate}
  For what follows, let $u=u_j$ be a point outside $X_{j-1}$ and for a subset $\sigma$ of $\mathfrak{q}(u)$, 
  define $\mathfrak{g}(\sigma;u):=\mathfrak{g}(u\cdot\prod_{y\in\sigma}y).$
Also, denote by $X(u)$
 the subcomplex of $X_{j-1}$, induced by $\{ \mathfrak{g}(\sigma;u) \}$ for all non-empty 
 subsets $\sigma$ of $\mathfrak{q}(u)$. We will show that $X(u)$ is an $(l-1)$-dimensional ball in 
 Lemma~\ref{subdivision} and Proposition~\ref{sh} below, where $l=\left|\mathfrak{q}(u)\right|$.
  
  \begin{enumerate}
  \item[(iii)] glue an $l$-ball $B(u)$ 
 along its boundary 
 to $X(u)\cup(\{u\}\ast\partial X(u))$.

 \end{enumerate}
 Having defined the new maximal cell $B(u)$, what remains in order to have a complete description of $X_j$ is to 
 give a cell decomposition for $\{u\}\ast\partial X(u)$. For this purpose let us for any proper subset $\sigma$
 of $\mathfrak{q}(u)$, denote by $X(\sigma,u)$ the 
 subcomplex of $X_{j-1}$ induced by the vertices $\{ \mathfrak{g}(\tau;u) \}$ for all non-empty 
 subsets $\tau$ of $\sigma$. The fact that $X(\sigma,u)$ is a ball of dimension $|\sigma|-1$ is needed for the 
 next step, we will discuss it in the end of this section in remark~\ref{xsigma}.
 \begin{enumerate}
  \item[(iv)] define $X_j$ to be $X_{j-1}\cup \{B(u)\}\bigcup_{\sigma\subset\mathfrak{q}(u)}\{B(\sigma, u) \}$,
  where $B(\sigma, u)$ denotes $\{u\}\ast X(\sigma,u)$.
 \end{enumerate}

 \end{cons}\qed
 \begin{rem}\label{facts} The following simple facts can be seen from the construction above:\\
 
\begin{enumerate}
  \item[(i)]For a non-empty subset $A$ of $G(I)$, the following are equivalent:
  
   \item[$\bullet$] There is a cell $B$ of $X_I$ whose set of vertices coincides with $A$,
   \item[$\bullet$] There is some $u\in G(I)$ and some $\sigma\subseteq \mathfrak{q}(u)$ such that 
   $$A=\{\mathfrak{g}(\tau;u)|\tau\subseteq\sigma\},$$
   \item[$\bullet$] $B=B(\sigma, u)$.\\
  Moreover, $B(\sigma, u)$ corresponds to the base element $f(\sigma;u)$ of the Herzog-Takayama resolution.
\item[(ii)] Maximal cells of $B(\sigma,u)$ are precisely:
\item[$\bullet$] $B(\sigma-\{y\},u)$ for all $y\in\sigma$, and
\item[$\bullet$] $B(\sigma-\{y\},\mathfrak{g}(yu))$ for those $y\in\sigma$ such that $\sigma-\{y\}\subseteq 
\mathfrak{q}(\mathfrak{g}(yu))$.
\end{enumerate}
 \end{rem}
\qed
 
 We now study some combinatorial properties of the cell complex $X_I$ constructed 
 in~\ref{cons}, via a simplicial subdivision of it. 
 We follow the notation in Construction~\ref{cons}. 
 Let us start by giving the simplicial subdivision of $X_I$. 
 
 \begin{cons}
   We construct a simplicial complex $\Lambda_I$ inductively as follows:
 \begin{enumerate}
  \item[(i)] Let $\Lambda_1$ be the $0$-simplex $\{u_1\}$,
  \item[(ii)] Assume that $\Lambda_{j-1}$ is constructed,
  \item[(iii)] Take a cone with apex $\{u\}$ over the subcomplex $\Lambda(u)$ of $\Lambda_{j-1}$, induced by 
  $\{ \mathfrak{g}(\sigma;u) \}$, for all subsets $\sigma$ of $\mathfrak{q}(u)$, to obtain 
  $$\Lambda_j=\Lambda_{j-1}\bigcup\left(\{u\}\ast\Lambda(u)\right).$$
 \end{enumerate}
 \end{cons}\qed

 \noindent The following is immediate. 
 \begin{lem}\label{subdivision}
  $\Lambda_j$ (respectively $\Lambda(u)$) is a simplicial subdivision of $X_j$ (respectively $X(u)$).\qed
  \end{lem}
  
 \noindent A \emph{closure operator} on a finite set $E$ is a function $c:2^E\rightarrow 2^E$, such that 
 for all 
 $\sigma,\tau\subseteq E$,
 \begin{enumerate}
  \item[(CO1)] $\sigma\subseteq c(\sigma)$,
  \item[(CO2)] $\sigma\subseteq\tau$ implies $c(\sigma)\subseteq c(\tau)$,
  \item[(CO3)] $c(c(\sigma))=c(\sigma)$.
 \end{enumerate}
 A closure operator $c$ is said to be an \emph{anti-exchange closure} if, in addition, for all $a\neq b$ in 
 $E$, it satisfies the following anti-exchange axiom:
 \begin{enumerate}
 \item[(AE)\hspace{0.08in}] If $a,b\notin c(\sigma)$ and $a\in c(\{b\}\cup\sigma)$, then $b\notin c(\{a\}\cup\sigma)$.
 \end{enumerate}
 A finite set $E$ together with an anti-exchange closure $c$ on it is called a \emph{convex geometry}. 
 See Bj\"{o}rner and Ziegler~\cite{BZ}, for a comprehensive introduction 
 to a more general topic, greedoids.\\ Now let us 
 define a closure operator on $\mathfrak{q}(u)$: For a subset $\sigma$ of $\mathfrak{q}(u)$ we let 
 $c(\sigma)$ to be the largest subset $\tau$ of $\mathfrak{q}(u)$ such that 
 $\mathfrak{g}(\sigma;u)=\mathfrak{g}(\tau;u)$. The fact that this operator is well defined follows from the 
 simple fact that for 
 any $\delta_1$ and $\delta_2$ if $\mathfrak{g}(\delta_1;u)=\mathfrak{g}(\delta_2;u)$, then 
 $\mathfrak{g}(\delta_1;u)=\mathfrak{g}(\delta_1\cup\delta_2;u)$.
 
 \begin{prop}\label{convex}
 The pair $\langle \mathfrak{q}(u), c\rangle$ is a convex geometry.
 \end{prop}
\begin{proof}
 The axioms (CO1) and (CO3) clearly hold. Assume that $\tau$ is the union $\delta\cup\sigma$. From the equation 
 $\mathfrak{g}(\delta\cup\sigma;u)=\mathfrak{g}(\delta ;\mathfrak{g}(\sigma;u))$, it follows that 
 $\mathfrak{g}(\tau;u)=\mathfrak{g}(\delta\cup c(\sigma);u)$ 
and in particular $c(\sigma)\subseteq c(\tau)$.\\
Now, we shall verify the anti-exchange axiom. Let $\sigma$ be any subset of $\mathfrak{q}(u)$ and denote by $v$ 
the monomial minimal generator $\mathfrak{g}(\sigma;u)$. The condition that $a\notin c(\sigma)$ is equivalent 
to $a\in\mathfrak{q}(v)$. Now for a variable $b$ in $\mathfrak{q}(v)$ different from $a$, let 
$a\in c(\{b\}\cup\sigma)$, or equivalently $\mathfrak{g}(abv)=\mathfrak{g}(bv)$. The fact that 
$\mathfrak{g}(av)$ is a minimal monomial generator different from $v$ that divides $av$ implies that 
$\mathfrak{g}(av)\neq\mathfrak{g}(bv)=\mathfrak{g}(abv)$, and therefore $b\notin c(\{a\}\cup\sigma)$.
\end{proof}

\noindent The set of all closed sets of a convex geometry, i.e. those subsets that are fixed under the 
closure operator, forms a 
 lattice when the partial ordering is inclusion. This lattice is known 
 to be meet-distributive, see e.g.~\cite[Theorem3.3]{E}. 
 \begin{prop}\label{sh}
  $\Lambda(u)$ is a shellable $(l-1)$-dimensional ball. 
 \end{prop}
\begin{proof}
 Let $L$ be the lattice of closed sets of $\langle \mathfrak{q}(u), c\rangle$ and $\widehat{L}=L\setminus
 \{\emptyset\}$. The decomposition function $\mathfrak{g}(.;u)$ can be seen as a bijection between the elements
 of $\widehat{L}$ and vertices of $\Lambda(u)$, that maps every closed set $\sigma$ to $\mathfrak{g}(\sigma;u)$. 
 Furthermore, via this map the faces of $\Lambda(u)$ are precisely chains in $\widehat{L}$ and hence 
 $\Lambda(u)$ is the order complex of $\widehat{L}$. Now, since the order complex of the proper part of any 
 meet-distributive lattice is a shellable ball (see, e.g., \cite{BHP}) , $\Lambda(u)$ is a shellable ball. The argument about dimension 
 follows from~\cite[Proposition 8.7.2.(iv)]{BZ}.
\end{proof}
\noindent A simplicial complex $K$ is said to be a \emph{PL--ball} if a subdivision of $K$ is a subdivision of 
a simplex. A regular cell complex $\Gamma$ is a PL--ball if its barycentric subdivision is a PL-ball. 
Many properties that we expect balls to have fails when they are not PL. For instance, the complex obtained by 
gluing two balls of the same dimension along a common codimension one ball lying on their boundaries is not 
necessarily a ball, if they are not PL. We refer to~\cite[Section 4.7.(d)]{BLSWZ} and the references therein 
for a detailed discussion on the subject.\\ We mentioned that for 
our construction to be well-defined we need $X(u)$ to be a ball, but actually the inductive construction 
requires more, that is   
\begin{cor}
 $X(u)$ is an $(l-1)$-dimensional PL--ball. 
\end{cor}
\begin{proof}
It follows from Lemma~\ref{subdivision}, Proposition~\ref{sh} and ~\cite[Corollary 2.2]{BHP}.

\end{proof}

\begin{rem}\label{xsigma}
Let $\sigma$ be a fixed subset of $\mathfrak{q}(u)$. For a subset $\delta$ of $\sigma$ define 
$c_{\sigma}(\delta)$ to be the largest subset $\tau$ of $\sigma$ such that 
$\mathfrak{g}(\delta;u)=\mathfrak{g}(\tau;u)$. Then, in a similar way as in Proposition~\ref{convex},
one can show that $\langle \sigma, c_{\sigma}\rangle$ is a convex geometry. In fact, the same line of reasoning
as we used to prove the above corollary implies that $X(\sigma,u)$ is a PL-ball of dimension 
$\mid\sigma\mid-1$.

\end{rem}\qed

 \section{Main Result}

\noindent In this section we prove our main result. We begin by proving two auxiliary lemmas. 

\begin{lem}\label{l1}
 $X_I$ is contractible. 
\end{lem}
\begin{proof}
 Let $f:X(u)\hookrightarrow X_{j-1}$ be the inclusion map. Then $X_j$ is homotopy equivalent 
 to the mapping cone $C_f$ of $f$. Now since $X(u)$ and $X_{j-1}$ are contractible, 
 by \cite[Page 27]{Sp} $f$ is 
 a homotopy equivalence and therefore $C_f$ is contractible (see, e.g., \cite[Section 4.2]{Ha} for more 
 details).
\end{proof}
\noindent Let $I\subset S$ be a monomial ideal and $\mu$ a monomial in $S$. Then we denote by $I_{\leq\mu}$
 the ideal generated by those monomials in $G(I)$ that divide $\mu$. 
\begin{lem}\label{l2}
 Let $I\subset S$ be a monomial ideal with linear quotient and let $\mu$ be a monomial in $S$. Then 
 $I_{\leq\mu}$ has linear quotient. Moreover if $I$ is regular, then so is $I_{\leq\mu}$.
\end{lem}
\begin{proof}
 Let $u_1,\ldots, u_m$ be an admissible order of minimal generators of $I$ and assume that $u_{i_1},\ldots ,u_
 {i_t}$ generates $I_{\leq\mu}$, where $i_1<\ldots <i_t$. We show that $u_{i_1},\ldots ,u_{i_t}$ is an 
 admissible order for $I_{\leq\mu}$. For $s<t$, the shelling type condition~\ref{shelling}, guarantees that 
 there exists $l<i_t$ such that $\mbox{lcm}(u_l,u_{i_t})=u_{i_t} x_r$ for some $x_r$ 
 and that $\mbox{lcm}(u_l,u_{i_t})$ divides $\mbox{lcm}(u_{i_s},u_{i_t})$. In particular $u_l$ divides $\mu$ and hence $u_l\in I_{\leq\mu}$.\\
 For the second part, denote by $\mathfrak{g}'$ the decomposition function of $I_{\leq\mu}$. Let $y$ be a 
 variable such that $\mathfrak{g}'(yu_{i_j})\neq u_{i_j}$ and $z$ another variable such that 
 $\mathfrak{g}'(z\mathfrak{g}'(yu_{i_j}))\neq \mathfrak{g}'(yu_{i_j})$. To show that $\mathfrak{g}'$ is 
 regular, it is enough to show that $\mathfrak{g}'(zu_{i_j})\neq u_{i_j}$. Assume not. Then for any 
 $l<i_j$ such that $u_l$ divides $zu_{i_j}$, $u_l$ does not divide $\mu$ and in particular $\mbox{deg}_z(u_{i_j})=
 \mbox{deg}_z(\mu)$, where $\mbox{deg}_z$ stands for the degree of $z$ in the monomial. On the other hand, 
 $v=\mathfrak{g}'(z\mathfrak{g}'(yu_{i_j}))$ divides $zyu_{i_j}$. So, 
 $v$ divides $yu_{i_j}$, since $\mbox{deg}_z(v)\leq \mbox{deg}_z(\mu)=\mbox{deg}_z(u_{i_j})$. This is a contradiction, 
 since $v=\mathfrak{g}'(z\mathfrak{g}'(yu_{i_j}))$ appears earlier than  $\mathfrak{g}'(yu_{i_j})$ in the 
 admissible order. 
 
\end{proof}

\noindent Now we are in position to prove our main result. 

\begin{thm}
 If $I$ has a regular linear quotient, then the labelled regular cell complex $X_I$ 
 supports the minimal free resolution of $I$.
\end{thm}
\begin{proof}
First observe that for any monomial $\mu$ the subcomplex $X_{\leq\mu}$ of $X_I$, consisting of all cells 
 with a label that divides $\mu$, is the same as the complex $X_{I_{\leq\mu}}$ and hence is contractible by 
 lemmas \ref{l1} and \ref{l2}.\\
Now we shall show that any two cells with a non trivial containment relation have different labels. 
Let $c$ be a cell of $X_I$. Then the vertices of $c$ are $v$ together with $\{ \mathfrak{g}(\sigma;v) \}
_{\sigma\subseteq\tau}$, for some $v$ and some $\tau\subseteq\mathfrak{q}(v)$ and in particular the label of 
$c$ is $v\prod_{z\in\tau}z$. For a maximal cell $c'$ of $c$, 
we have only two possibilities: either $c'$ contains $v$ or not. In the former case the other vertices of 
$c'$ are $\{ \mathfrak{g}(\sigma;v) \}
_{\sigma\subseteq\tau'}$, where $\tau'=\tau\setminus \{y\}$ for some variable $y$. So, the label of $c'$ is 
different from the label of $c$. In the latter case, the vertices of $c'$ are $\mathfrak{g}(yv)$ together with 
$\{ \mathfrak{g}(\sigma;\mathfrak{g}(yv)) \}_{\sigma\subseteq\tau'}$, for some $y\in\mathfrak{q}(v)$,
where $\tau'=\tau\setminus \{y\}$. The label of $c'$ is then $\mathfrak{g}(yv)\prod_{z\in\tau'}z$ which is 
different from the label of $c$, since $\mbox{deg}(\mathfrak{g}(yv))\leq \mbox{deg}(v)$ and $\left|\tau' \right|<
\left|\tau \right|$.\\

\end{proof}

The bijection between cells of the cellular complex $X_I$ and the base elements in the Herzog-Takayama 
resolution $\mathbf{F}_I$ (Remark~\ref{facts}) shows that the cellular resolution obtained from our construction 
coincides with the Herzog-Takayama resolution. Thus, we have the folloewing:

\begin{thm}
 Let $I$ be an ideal with a regular linear quotient. Then the Herzog-Takayama resolution $\mathbf{F}_I$ is
 a minimal cellular resolution of $S/I$.
\end{thm}\qed

\paragraph{\textbf{Open Problem.}} Batzies and Welker~\cite[Proposition 4.3]{BW} proved that for a monomial 
ideal $I$ with a linear quotient there exists a minimal free Lyubeznik-resolution\footnote{The reader should be aware that Batzies and Welker's 
Lyubeznik-resolution is a generalization of the one constructed by G.~Lyubeznik. The generalized 
Lyubeznik-resolution is not necessarily simplicial, whereas the Lyubeznik's resolution always is.} of $S/I$, 
and thus there exists 
a CW complex that supports the resolution of $I$ (see~\cite{BW}, for more details). When 
$I$ has a regular linear quotient, it 
would be interesting to determine whether our complex $X_I$ necessarily coincides with the one coming from the 
Lyubeznik-resolution or not.

\subsection*{Acknowledgments}
 I would like to thank Anders Bj\"{o}rner for helpful discussions especially for bringing convex geometries
 to my attention, Volkmar Welker for some very helpful comments. I would also like to thank the anonymous 
 referee for many helpful sugesstions that led to improvements of the exposition.


\begin{thebibliography}{10}

\bibitem{BW} E.~Batzies, V.~Welker, {\it Discrete Morse theory for cellular resolutions},
J. Reine Angew. Math. {\bf 543} (2002), 147--168.
\bibitem{BPS} D.~Bayer, I.~Peeva, B.~Sturmfels, {\it Monomial resolutions}, Math. Res. Lett. {\bf 5} (1998), 
31--46. 


\bibitem{BS}D.~Bayer and B.~Sturmfels, {\it Cellular resolutions of monomial modules}, 
J. Reine Angew. Math. {\bf 502} (1998), 123--140.

\bibitem{BHP} L.J.~Billera, S.K.~Hsiao, J.S.~Provan, {\it Enumeration in convex geometries and associated
polytopal subdivisions of spheres}, Discrete Comput. Geom. {\bf 39} (2008) 123--137.

\bibitem{BLSWZ} A.~Bj\"{o}rner, M.~Las Vergnas, B.~Sturmfels, N.~White, G.M.~Ziegler,  
Oriented Matroids. 2nd ed. Cambridge University Press, Cambridge, (1999).

\bibitem{BZ}  A.~Bj\"{o}rner, G.M.~Ziegler, Introduction to greedoids, in: N. White (Ed.),
Matroid Applications, Cambridge University Press, (1992), 284--357.

\bibitem{E} P.H.~Edelman, {\it Meet-distributive lattices 
and the anti-exchange closure}, Algebra Universalis {\bf 10}, 290--299 (1980).

\bibitem{Ha} A.~Hatcher, Algebraic topology, Cambridge University Press, Cambridge, (2002).
\bibitem{HT} J.~Herzog, Y.~Takayama, {\it Resolutions by mapping cones}, in: The Roos Festschrift, vol. 2,
Homology, Homotopy Appl. 4 (2, part 2) (2002) 277--294.

\bibitem{JZ} A.S.~Jahan, X.~ Zheng, {\it Ideals with linear quotients}, J. Combin. Theory. Ser. A.
{\bf 117} (1), (2010), 104--110.

\bibitem{M} J.~Mermin, {\it The Eliahou-Kervaire resolution is cellular}, J. Commut. Algebra {\bf 2} (1)
(2010) 55--78.

\bibitem{MS}E.~Milller, B.~Sturmfels, Combinatorial Commutative Algebra, Grad. Texts in Math., vol. 227, 
Springer-Verlag, New York, (2005).

\bibitem{NPS} I.~Novik, A.~Postnikov, B.~Sturmfels, {\it Syzygies of oriented matroids}. Duke Math. J.
{\bf 111} (2002), no. 2, 287--317.

\bibitem{P} I.~Peeva, Graded syzygies. Springer, (2010).

\bibitem{RW} V.~Reiner, V.~Welker, {\it Linear syzygies of Stanley--Reisner ideals}, Math. Scand.
{\bf }89, No. 1 (2001), 117--132.

\bibitem{Sp} E.~H.~Spanier, Algebraic topology, Springer-Verlag, New York, (1981). 

\bibitem{V} M.~Velasco, {\it Minimal free resolutions that are not supported on a CW-complex},
Journal of Algebra {\bf 319} (1) (2008) 102--114.

\end{thebibliography}
\end{document}